\documentclass[a4paper,12pt]{amsart}
\usepackage[public]{optional}

\opt{private}{\newcommand{\private}[1]{{#1}}}
\opt{public}{\newcommand{\private}[1]{}}
\opt{private}{\usepackage[notref,notcite]{showkeys}}
\newcommand{\PL}[1]{\private{\footnote{\textbf{Pascal: }{#1}}}}

\usepackage{amsmath,amssymb,amsfonts,amsthm}
\usepackage{endnotes}
\usepackage{fancyhdr}
\usepackage{lmodern}
\usepackage[utf8]{inputenc}
\usepackage[T1]{fontenc}
\usepackage{enumerate}
\usepackage[all]{xy}
\usepackage{graphicx}
\usepackage[english]{babel}

\newdir{ >}{{}*!/-5pt/\dir{>}}
\newtheorem{theo}{Theorem}[section]    
\newtheorem{de}[theo]{Definition}
\newtheorem{prop}[theo]{Proposition}
\newtheorem{lem}[theo]{Lemma}    
\newtheorem{cor}[theo]{Corollary}

\newtheorem{rem}[theo]{Remark}

\newcommand*{\defeq}{\mathrel{\vcenter{\baselineskip0.5ex \lineskiplimit0pt
                     \hbox{\scriptsize.}\hbox{\scriptsize.}}}%
                     =}
\newcommand{\delk}{\partial_W K}

\newcommand{\dela}{\partial A}
\newcommand{\delb}{\partial B}
\newcommand{\delt}{\partial T}
\newcommand{\delpt}{\partial_+ T}
\newcommand{\delmt}{\partial_- T}
\newcommand{\delot}{\partial_0 T}

\newcommand{\apl}{A_{PL}}
\newcommand{\proj}{\operatorname{proj}}

\newcommand{\Conf}{\operatorname{Conf}}
\newcommand{\btb}{B\otimes B}

\newcommand{\delw}{\partial W}
\newcommand{\delW}{\partial W}
\newcommand{\wst}{\overline{W\backslash T}}
\newcommand{\wsk}{W\backslash K}

\newcommand{\bq}{\mathbb{Q}}
\newcommand{\br}{\mathbb{R}}
\newcommand{\BR}{\mathbb{R}}

\newcommand{\quism}{\stackrel{\simeq}{\longrightarrow}}

\newcommand{\hoker}{\text{hoker }}
\newcommand{\troncan}{\tau^{\leq N} A}

\newcommand{\sn}{s^{-n}\#}
\newcommand{\sdn}{s^{-2n}\#}

\newcommand{\bs}{\setminus}

\newcommand{\wh}{\widehat{\apl}}

\newcommand{\ahatv}{(\hat{A}\otimes \Lambda V,D)}

\newcommand{\ahat}{\hat{A}}

\newcommand{\troncon}{\tau^{\leq N} C(\varphi^!)}
\newcommand{\tronc}{\tau^{\leq N}}
\newcommand{\Tronc}{\tau^{\leq N}}
\def\commutatif{\ar@{}[rd]|{\circlearrowleft}}
\def\pullback{\ar@{}[rd]|{\textrm{pullback}}}
\def\comhotopie{\ar@{}[rd]|{\sim}}
\newcommand{\oid}{\operatorname{id}}

\newcommand{\bth}{\bar\theta}

\newcommand{\qi}{\stackrel{\simeq}\longrightarrow}
\newcommand{\betab}{\xymatrix{\beta \colon B \ar@{->>}[r]& \delb}}

\newcommand{\BQ}{\mathbb{Q}}

\newcommand{\id}{\operatorname{id}}

\begin{document}

\title[Rational model of a complement]{Rational models of the complement \\ of a subpolyhedron in a manifold \\ with boundary}
\author{Hector Cordova Bulens}
\address{H.C.B. and P.L.: IRMP, Universit\'e catholique de Louvain, 2 Chemin du Cyclotron, B-1348 Louvain-la-Neuve, Belgium}
\email{hector.cordova@uclouvain.be}
\author{Pascal Lambrechts}
\email{pascal.lambrechts@uclouvain.be}
\author{Don Stanley}
\address{D.S.: University of Regina, Department of Mathematics\\
College West, Regina, CANADA}%
\email{stanley@math.uregina.ca}%
\thanks{H.C.B. and P.L. are supported by the belgian federal fund PAI Dygest}%
\date{18 May 2015}

\begin{abstract}
Let $W$ be a compact simply connected triangulated manifold with boundary and $K\subset W$ be a subpolyhedron.  We construct an algebraic model of the rational homotopy type of $\wsk$ out of a model of the map of pairs $(K,K \cap \delw)\hookrightarrow (W,\delw)$ under some high codimension hypothesis.

We deduce the rational homotopy invariance of the configuration space
of two points in a compact manifold with boundary under 2-connectedness hypotheses. Also, we exhibit
nice explicits models of these configuration spaces for a large class
of compact manifolds.

\end{abstract}
\maketitle

\section{Introduction}

\bigskip

Let $W$ be a compact and simply-connected manifold with boun\-dary (in
this paper all manifolds are triangulated). Let $f\colon K
\hookrightarrow W$ be the inclusion of a subpolyhedron. The first goal
of this paper is to determine the rational homotopy type of the complement
$\wsk$.  We will then apply this to deduce the rational homotopy
type of the configuration space of two points in a manifold with
boundary under 2-connectedness hypotheses. Hence this paper extends
the results of \cite{LaSt:PE}
and \cite{LaSt:FM2}
to the case of manifolds \emph{with boundary}.

The main result of \cite{LaSt:PE} is an explicit description of the rational homotopy type of $\wsk$ when $W$ is a \emph{closed} manifold and $K$ is a subpolyhedron of codimension  $ \geq(\dim W)/2 +2$. This rational homotopy type depends only on the rational homotopy class of the inclusion $K \hookrightarrow W$ (\cite[Theorem 1.2]{LaSt:PE}).

The situation for manifolds with boundary is different. For example,
let $W$   be an $n$-dimensional disk $D^n$ and $K$ be a point. If $K$ is
embedded in the interior of $D^n$ then $\wsk \simeq S^{n-1}$. On the
contrary, if $K$ is embedded in the boundary of $D^n$ then $\wsk
\simeq *$.
Hence the complements $\wsk$ have different rational homotopy types, 
 although the two inclusions $K\hookrightarrow W$ are homotopic.
    These examples show that we need more information to determine the
    rational homotopy type of $\wsk$. Our main result is that the only
    extra information needed is related to the inclusion of $\delw
    \cap K$ in $\delw$. More precisely, we have the following result
\begin{theo}[Corollary \ref{C:mod_conndim} and Theorem
  \ref{T:mod_conndim}]
\label{T:intro1}
Let $W$ be a compact simply connected triangulated manifold with boundary and let $K$ be a subpolyhedron in $W$. 

Assume that 
\begin{equation}
\label{e:codim_intro}
\dim W \geq 2\dim K + 3.
\end{equation}
Then the rational homotopy type of $\wsk$ depends only on the rational homotopy type of the square of inclusions 
\begin{equation}
\label{d:intro}
\xymatrix{ (K\cap \delw) \ar@{^{(}->}[r]\ar@{^{(}->}[d] & \delw\ar@{^{(}->}[d] \\
K \ar@{^{(}->}[r] & W.}
\end{equation}
Moreover a CDGA model of $\wsk$  (that is, an algebraic model in the
sense of Sullivan of this rational homotopy type, see Section \ref{s:RHT}) can be explicitely constructed out of any CDGA model of Diagram \eqref{d:intro}.
\end{theo}

Actually we will see that the high codimension hypothesis
\eqref{e:codim_intro} can be weakened. Indeed 
we will establish a sharp \emph{unknotting condition}, which
is an inequality relating the connectivity of the inclusion maps and the
dimensions of the manifold and the subpolyhedron (see \eqref{unknot2}
in Corollary \ref{C:mod_conndim}), 
under which we still
get a CDGA model of the complement.

There is an interesting application of this theorem to the study of
confi\-guration spaces of 2 points in $W$,
\[\Conf(W,2) \defeq \{(x_1, x_2) \in W\times W \colon x_1 \neq x_2\}.\]
 Indeed this configuration space is the complement
\[\Conf(W,2) = W\times W \bs \Delta (W)\]
where $\Delta \colon W \hookrightarrow W\times W$ is the diagonal
embedding. We will deduce from Theorem \ref{T:intro1}  the following result. 
\begin{theo}[Corollary \ref{c:rht_fw2}]
\label{T:intro2}
Let $W$ be a $2$-connected compact mani\-fold with a $2$-connected or
empty boundary. The rational homotopy type of the configuration space 
$\Conf(W,2)$
depends only on the rational homotopy type of the pair $(W,\delw)$. 
\end{theo}

In \cite{CLS:pretty} we prove that a large class of compact manifolds
with boundary admit CDGA models of a special form that we call 
\emph{surjective pretty models}. This class contains in particular 
even-dimensional
disk bundles over a closed manifold and complements of  high
codimensional polyhedra in  closed manifolds.
As a consequence, such manifolds admit a
CDGA model of the form $P/I$ where $P$ is a Poincar\'e duality CDGA
and $I$ is some differential ideal. Poincar\'e duality CDGAs come with a
natural \emph{diagonal class} $\Delta \in (P\otimes P)^n$.  We then
get the following elegant model for the configuration space (see
Section \ref{S:ConfPretty} for more details)
\begin{theo}(Theorem \ref{T:modele_FW2}) 
\label{T:introFW2}
  Let $W$ be a compact
  mani\-fold of dimension $n$ with boundary and assume that $W$ and $\delw$ are
  $2$-connected.
 If $(W,\delw)$ admits a surjective pretty model in the sense of
 \cite{CLS:pretty} then a CDGA model of $\Conf (W, 2)$ is given by 
\[\left(P/I \otimes P/I \right)\oplus_{\overline{\Delta^!}} ss^{-n} P/I,\]
where $P$ is the Poincar\'e duality CDGA and $I$ the ideal associated
to the pretty model, and 
$\overline{\Delta^!}$ is a map induced by multiplication by
the diagonal class $\Delta\in (P\otimes P)^n$.
\end{theo}

When $W$ is a closed manifold, we have $I=0$ and the model of Theorem
\ref{T:introFW2}
is exactly that of \cite{LaSt:FM2}.

In the paper \cite{CLS:ConfWk} in preparation we will show how to
build a model (of dgmodules) of $\Conf(W,k)$, $k\geq2$, which 
enables to compute effectively the homology of the space of configurations of
any number of points in a manifold with boundary. This model will be
of the form
\[\left(\frac{(P/I)^{\otimes k}\otimes\Lambda(g_{ij}:1\leq i<j\leq
    k)}{\textrm{(Arnold and symmetry
      relations)}}\,,\,d(g_{ij})=\pi_{ij}^*(\overline{\Delta})\right),
\]
mimicking the model in \cite{LaSt:remFMk}.

Here is the plan of the paper. Section \ref{s:map_con} contains a very short
review on rational homotopy theory, the notion of truncation of a
CDGA, a discussion on CDGA structures on mapping cones, and the notion
of homotopy kernel. Section \ref{s:lef_dual} is a first step to the
understanding of a dgmodule model of the complement $W\setminus K$ and
in  Section \ref{s:complement} we establish a CDGA model of that complement.
In Section \ref{S:ConfW2} we apply the previous results to the model of
 the configuration space of $2$ points in compact manifolds, with
 some developments of the examples of configuration spaces on a disk
 bundle or in the complement of a polyhedron in a closed manifold.

\section{Truncation of dgmodules and CDGA's, and
    CDGA structures on mapping cones.}
\label{s:map_con}
This section contains a quick review on some classical topics that we
will need with some special development. In particular in
\ref{S:truncation_cdga}
we explain some notion of truncation of a CDGA, and in
\ref{s:semi_trivial} we show how to endow a mapping cone (or
its truncation) with the structure of a CDGA.

\subsection{ Rational homotopy theory} \label{s:RHT}%
In this paper we will use the standard tools and results of rational
homotopy theory,  following \cite{FHT:RHT}. Recall that $\apl$ is
the Sullivan-de Rham functor and that for a 1-connected space of
finite type, $X$,
$\apl(X)$ is a commutative differential graded algebra (CDGA for
short), which completely encodes the rational homotopy type of $X$.  
Any CDGA  weakly equivalent to $\apl(X)$ is called a \emph{CDGA model
  of $X$}. All our dgmodules and CDGAs are over the field $\bq$.

\subsection{Truncation of a dgmodule}

The classical truncation of a co\-chain complex, i.e. $\bq$-dgmodule, $C$, is classicaly defined by (see \cite[Section 1.2.7]{Wei:HA}) 
\begin{equation}
\label{Eq:tronWeibel}
(\hat\tau^{\leq N} C)^i = 
\begin{cases} 
C^i & \text{ if } i<N\\
C^N \cap \ker d & \text{ if } i=N \\
0& \text{ if } i >N 								
\end{cases}
\end{equation}
This comes with an inclusion 
\[\iota \colon \hat\tau^{\leq N} C \hookrightarrow C\]
which induces isomorphisms $H^i(\iota)$, for $i\leq N$, and such that $H^{>N}$ $(\hat\tau^{\leq N} C) =0$.

When $R$ is an $A$-dgmodule, the   truncation $\hat\tau^{\leq N} R$ is not necessarily an $A$-dgmodule. In that case a better replacement would be to take for the truncation a quotient $R/I$ where $I$ is a suitable $A$-dgsubmodule such that $I^i=R^i$ for $i>N$. In this paper we will use the following:
\begin{de}
Let $R$ be an $A$-dgmodule and let $N$ be a positive integer. A \emph{truncation below degree $N$ of $R$} is an $A$-dgmodule, $\Tronc R$, and a morphism $\pi \colon R \to~\Tronc R$ of $A$-dgmodules verifying the two following conditions:
\begin{enumerate}
\item $(\Tronc R)^{>N} = 0$ and $(\Tronc R)^{<N}\cong R^{<N}$, and
\item the morphism $\pi$ is a surjection of $A$-dgmodules such that $H^i(\pi)$ is an isomorphism for $0\leq i \leq N$.
\end{enumerate}

\end{de}

Contrary to $\hat \tau^{\leq N}$ from \eqref{Eq:tronWeibel}, our
truncation $\Tronc R$ is not unique and is not a functorial
construction.

\subsection{Truncation of a CDGA}\label{s:truncation_cdga}
\label{S:truncation_cdga}

\begin{de}
Let $A$ be a connected CDGA.  \emph{A CDGA truncation below degree $N$ of $A$}  is a truncation of $A$-dgmodule $ (\Tronc A, \pi)$  such that $\Tronc A$ is a CDGA and $\pi:A\to \hat\tau^{\leq N} A$ is a CDGA morphism.
\end{de}

Equivalently a CDGA truncation can be seen as a projection $\pi \colon A \to A/I$ where $I$ is an ideal of $A$ such that $I^{<N}=0$, $I^{>N}=A^{>N}$ and $I^N\oplus (\ker d \cap A^N)= A^N$. 

\begin{prop}
Any two CDGA truncations below degree $N$ of a given connected CDGA are weakly equivalent.
\end{prop}
\begin{proof}

Let $A$ be a connected CDGA and $N\in \mathbb{N}$.
It is easy to construct a relative Sullivan model 
\[\xymatrix{\iota \colon A \ar@{ >->}[r] & (A\otimes \Lambda V, D)}\]
such that $H^{\leq N} (\iota)$ is an isomorphism, $V=V^{\geq N}$ and $H^{>N} (A\otimes \Lambda V,D)=0$. Indeed, one builds inductively $V=V^{\geq N}$ by adding generators to eliminate all the homology in degrees $>N$. It is straightforward to check that any CDGA truncation $\pi \colon A \to \tronc A$ factors as follows
\[\xymatrix{A \ar[rr]^{\pi} \ar@{>->}[dr]_{\iota} &&\troncan \\
& (A\otimes \Lambda V, D) \ar[ur]_{m} &}\]
where $m(V)=0$. Since $H^{\leq N} (\iota)$ and $H^{\leq N} (\pi)$ are isomorphisms and $H^{>N}(A\otimes \Lambda V, D)= H^{>N}(\troncan)=0$, we deduce that $m$ is a quasi-isomorphism. Therefore any two truncation of $A$ are quasi-isomorphic to $(A\otimes \Lambda V, D)$, and hence are weakly equivalent 

\end{proof}

\subsection{Semi-trivial C(D)GA structures on mapping cones}
\label{s:semi_trivial}
Let $A$ be a CDGA and 
 let $R$ be an $A$-dgmodule. We will denote by $s^k R$ the k-th suspension of $R$, i.e. $(s^k R)^p= R^{k+p}$, and for  a map of $A$-dgmodules, $f\colon R \to Q$,  we denote by $s^k f$ the  k-th suspension of  $f$. Furthermore, we will use $\#$ to denote the linear dual of a vector space, $\#V=hom(V,\mathbb{Q})$, and $\#f$ to denote the linear dual of a map $f$.  

If $f\colon Q \to R$ is an $A$-dgmodule morphism, \emph{the mapping cone} of $f$ is the $A$-dgmodule 
\[
C(f)\defeq (R\oplus_f sQ, \delta)
\]
defined by $R\oplus sQ$ as an $A$-module and with a differential
$\delta$ such that $\delta (r,sq) = (d_R (r) +f(q), -sd_Q(q))$. 

When
$R=A$, the mapping cone $C(f\colon Q\to A)$  can be equipped with a unique
commutative graded algebra (CGA) structure that extends the algebra
structure on $A$, respects  the $A$-dgmodule structure,  and such
that $(s q) \cdot (sq')=0$, for $q,q' \in Q$. We will call this
structure the \emph{semi-trivial CGA structure} on the mapping cone
$A\oplus_fsQ$ (see \cite[Section 4]{LaSt:PE}).      The following result is very useful to detect when this CGA structure is, in fact, a CDGA structure.

\begin{de}
\label{D:balanced}
Let $A$ be a CDGA. An $A$-dgmodule morphism $f\colon Q \to A$ is \emph{balanced} if :
\begin{equation}
f(x)y=xf(y) \ \ \text{    for all $x,y\in Q$.}
\label{Eq:anodyne}
\end{equation}
\end{de}
The importance of this notion comes from the following proposition.
\begin{prop}
\label{p:nagata}
Let $Q$ be an $A$-dgmodule and $f\colon Q \to A$ be an $A$-dgmodule morphism. If $f$ is balanced  then the mapping cone $C(f)=A\oplus_f sQ$ endowed with the semi-trivial CGA
structure is a CDGA. 
\end{prop}
\begin{proof}
The only non trivially verified condition for $C(f)$ being a CDGA is the Leibniz rule for the differential,  which is a consequence of $(\ref{Eq:anodyne})$.    See the proof of \cite[Proposition 2.2]{CLS:pretty} for more details.
\end{proof}

 \begin{prop}\label{p:truncMC}
Let $A$ be a connected CDGA and let $f\colon Q\to A$ be an $A$-dgmodule morphism.   Let $p$ and $N$ be natural integers such that $Q^{<p}=0$ and $N\leq 2p-3$.    Then the semi-trivial CGA structure on the mapping cone $C(f)$ induces a CDGA structure on $\tau^{\leq N} (C(f))$, and 
\[A\to \tau^{\leq N} (C(f))\]
is a CDGA morphism. 
\end{prop}

\begin{proof}
Analogous argument as for the proof of Proposition 2.5.  See the proof of  \cite[lemme 4.5]{LaSt:PE} for more details.  
\end{proof}

\begin{rem}
 In the rest of this paper, when a mapping cone is equipped with a CDGA structure it will be understood that it comes from the semi-trivial structure.
\end{rem}

\subsection{Homotopy kernel}
\label{S:hoker}
In this section we recall the notion of homotopy kernel and  some of its properties. 

\begin{de}
Let $f \colon M \to N$ be a morphism of $A$-dgmodules. The \emph{homotopy kernel} of $f$ is the $A$-dgmodule mapping cone
\[\mathrm{hoker} \  f \defeq s^{-1}N\oplus_{s^{-1}f} M,\]
which comes with an obvious map 
\[\mathrm{hoker} \  f \to M \ ; \  (s^{-1}n, m) \mapsto m.\]
\end{de}

The following result is a consequence of the five lemma and  justifies the terminology ``homotopy kernel''. 
\begin{prop}
 Let $f\colon M \to N$ be a surjective morphism of $A$-dgmodules.  Then the morphism
  \[
  \begin{array}{cccc}
  \varphi : &\ker f &\qi &  \mathrm{hoker} \  f \\
& m &\longmapsto & (0, m)\end{array}\]
is an $A$-dgmodule quasi-isomorphim. 
\end{prop}

\section{Lefschetz duality for manifolds with boundary}
\label{s:lef_dual}

\bigskip

The aim of this section is to prove Proposition \ref{modele_apl_bord}
below, which is a first step towards the description of the rational homotopy type of the complement of a subpolyhedron in a manifold with boundary.

Let $W$ be a closed connected oriented triangulated manifold of dimension $n$ with boundary and let  $f: K\hookrightarrow W$ be the inclusion of a connected  subpolyhedron of dimension $k$ in $W$. Denote by $\delw$ the boundary of  $W$ and set  
\begin{equation}
\label{E:dWK}
\partial_W K := K\cap \partial W.
\end{equation}
In this section we will  construct a dgmodule model of $W\setminus K$, extending \cite[Theorem 6.3]{LaSt:remFMk} to manifolds with boundary. 
Consider the diagram
\begin{equation}
\label{diag1}
\xymatrix{
W &\ar@{_{(}->}_{f}[l] K   \\
\delw  \ar@{_{(}->}[u]   & \ar@{_{(}->}^{\partial f}[l] \delk,\ar@{_{(}->}[u], }
\end{equation}
which after applying the $\apl$ functor gives
\begin{equation}
\label{diag2}
\xymatrix{
\apl(W) \ar[r]^-{\apl(f)}  \ar[d]  & \apl(K)\ar[d]  \\
\apl(\delw)\ar[r] &\apl(\delk). }
\end{equation}
Recall that for a map of spaces $Y\to X$, we set 
\[\apl(X,Y)= \ker (\apl(X) \to \apl(Y)).\]
The inclusion of pairs
\[i\colon(K,\delk)\hookrightarrow (W,\delW)\]
induces an $\apl(W)-$dgmodule morphism
\begin{equation}
\label{E:apli}
\apl(i) \colon \apl(W, \delw) \to \apl(K,\delk).
\end{equation}
Using our notation for mapping cones, suspension and linear duals
from Section \ref{s:semi_trivial}, consider the map
\[
s^{-n}\# \apl(i) \colon \sn \apl(K, \delk) \to \sn\apl(W,\delw)
\]
and its mapping cone
\begin{equation}
\label{E:MCs-nApli}
C\left(s^{-n}\#
  \apl(i)\right)\,\,=\,\,\sn\apl(W,\delw)\,\oplus_{s^{-n}\# \apl(i)}\,s
\sn \apl(K, \delk) 
\end{equation}
with the inclusion
\begin{equation}
\label{E:iotaCsapli}
\iota\colon 
\sn\apl(W,\delw)
\hookrightarrow
C\left(s^{-n}\#
  \apl(i)\right).
\end{equation}
Since $(W,\delw)$ is an oriented compact manifold of dimension $n$, Poincar\'e
duality induces
a quasi-isomorphism of $\apl(W)-$dgmodules
\begin{equation}
\label{E:PhiW1}
\Phi_W\colon\apl(W)\qi s^{-n}\#\apl(W,\delw)
\end{equation}
(see $(\ref{E:PhiW})$ in the proof of Proposition
\ref{modele_apl_bord} below for an explicit
description of $\Phi_W$.)

\begin{prop}
\label{modele_apl_bord}
The map 
\[\apl(W)\to\apl(\wsk)\]
is weakly equivalent in the category of $\apl(W)$-dgmodules to the map
\[
\iota\circ\Phi_W\colon\apl(W)\to C(s^{-n}\# \apl(i))\]
where $C(s^{-n}\# \apl(i))$ is the mapping cone \eqref{E:MCs-nApli},
$\iota$ is from \eqref{E:iotaCsapli}, and $\Phi_W$ is from \eqref{E:PhiW1}.
\end{prop}

\begin{proof}
First we review from \cite[Section 4]{LaSt:remFMk}  a variation of the functor $\apl$ defined on
ordered simplicial complex and having an improved excision property.
Recall from \cite[Chapter 10]{FHT:RHT} that $\apl$ is actually defined
first on simplicial sets. Consider the category,  ${\mathcal K}$, of
ordered simplicial complexes.
To any ordered simplicial complex, $K$, we can associate naturally a
simplicial set, $K_\bullet$, whose non-degenerate simplices 
 are exactly the simplices of $K$ (see \cite[p.108]{Cur:SHT}).
Define the functor  
\[\widehat{\apl} \colon {\mathcal K} \to ADGC ; K \to \apl (K_\bullet).\]
This functor verifies the two following properties (see \cite[Section 4]{LaSt:remFMk}):
\begin{enumerate}
\item $\apl (| K|) \simeq\widehat{\apl} (K)$ naturally for every ordered
  simplicial complex (where $|K|$ is the geometric realization).
\item \underline{\emph{Strong excision property}} : Let $(K,L)$ be a pair of ordered simplicial complexes. Let $K'\subset K$ a sub-complex and $L'= K'\cap L$. If $K'\cup L= K$ then the inclusion $j \colon (K', L') \hookrightarrow (K, L)$ induces an \emph{isomorphism}
\[\widehat{\apl}(j) \colon \widehat{\apl}(K,L)
\stackrel{\cong}\longrightarrow \widehat{\apl}(K', L').\]
(Note that $\apl(j)$ is a quasi-isomorphism by the classical excision property.)
\end{enumerate}

Consider now the triangulated compact manifold $W$ and its
subpolyhedron $K$.
Replace those polyhedra $W$ and $K$ by their second barycentric
subdivision. Denote by $T$ the star of $K$ in $W$,
 which is a regular neighborhood (see \cite[chapters 1 and
 2]{Hud:PLT}), hence $T$ is a codimension 0 submanifold with boundary and it retracts by deformation onto $K$. It is clear that the topological closure  $\overline{W\bs T}$ of ${W\bs T}$ is homotopy equivalent to $\wsk$.  Set \begin{eqnarray*}\delpt&=&\delt\cap \delw  \\ \delmt&=&\overline{ (\delt \cap (W\bs \delw))}=T\cap \overline{W\bs T} \\ \partial_0 T &=& \delpt\cap \delmt,
\end{eqnarray*} which gives a decomposition of the boundary of
$T$, \[\delt = \delpt \cup_{\delot} \delmt.\] 

Our next goal is to set up Diagram \eqref{D:big} below.
Let us fix an arbitrary order on the vertices of the simplicial complex $W$ such that $W$ and the subpolyhedron $T$, $\delt$, $\delpt$, $\delmt$, $\delot$, $K$ and $\delk$ turn into ordered simplicial complexes. We can apply to them the functor $\wh$ which is naturally quasi-isomorphic to $\apl$.
 To prove the result, it suffices to show that the mapping cone
 $C(\sn$ $\wh(i))$ is a model of $\wh(W)$-dgmodule of $\wh(\wst)$.  To
 ease notations, in the rest  of this proof we will write $\apl$ instead of $\wh$.

The inclusion of the pair 
\[(T,\delt) \hookrightarrow (W, \overline{W\backslash T} \cup \delw)\]
induces by the strong excision property above an isomorphism
\[\apl(W,\overline{W\backslash T}\cup \delw) \stackrel{\cong}\rightarrow \apl(T,\delt).\]
Denote by $n$ the dimension of $W$. By Poincaré duality of the pair $(W,\delw)$, there exists an orientation \[\epsilon_W : \apl(W,\delw) \to s^{-n}\mathbb{Q},\]
i.e. a morphism of cochain complexes that induces an isomorphism in cohomology in degree $n$. Using this morphism we can define a morphism of $\apl(W)$-dgmodules
\begin{equation}\label{E:PhiW}
\begin{array}{cccc}\Phi_W \colon& \apl(W) &\longrightarrow &\sn \apl(W,\delw)\\
& \alpha & \longmapsto & \left(\Phi_W(\alpha)\colon \beta \mapsto
  \epsilon_W (\alpha \beta)\right), \end{array} 
\end{equation}
which is a quasi-isomorphism by Poincaré duality of the pair $(W,\delw)$.
The composition 
\[\epsilon_T: \apl(T,\delt)\cong\apl(W, \overline{W\backslash T} \cup \delw)\stackrel{\apl(incl)}\longrightarrow \apl(W,\delw) \stackrel{\epsilon_W}{\longrightarrow} s^{n}\mathbb{Q}\]
induces an isomorphism in cohomology in degree $n$. Define
\[\begin{array}{cccc}\Phi_T \colon& \apl(T) &\longrightarrow &\sn \apl(T,\delt)\\
& \alpha & \longmapsto & \left(\Phi_T(\alpha)\colon \beta \mapsto \epsilon_T (\alpha \beta)\right) \end{array} \]
which is a quasi-isomorphism of $\apl(W)$-dgmodules by Poincaré duality of the pair $(T,\delt)$.
Also, using the quasi-isomorphism above and the five lemma, it is not difficult to see that the morphism
\[\begin{array}{cccc}\tilde\Phi_T \colon& \apl(T,\delmt) &\longrightarrow &\sn \apl(T,\delpt)\\
& \alpha & \longmapsto & \left(\tilde\Phi_T(\alpha)\colon \beta \mapsto \epsilon_T (\alpha \beta)\right) \end{array} \]
is \  a  \  quasi-isomorphism  \  of  \  $\apl(T)$-dgmodules,  \  hence \   of \   $\apl(W)$-dgmodules.

The inclusion
\[(K,\delk) \hookrightarrow (T,\delpt)\]
is a homotopy equivalence and induces a weak equivalence of $\apl(W)$-dgmodules
\[\apl(T,\delpt) \qi \apl(K,\delk).\]
By the strong excision property, the inclusion  
\[(T,\delmt) \hookrightarrow (W, \wst)\]
induces an isomorphism \[\apl(W, \wst) \stackrel{\cong}\rightarrow \apl(T,\delmt).\] 
Combining all these morphisms we get the following commutative diagram of $\apl(W)$-dgmodules
\begin{equation}
\label{D:big}
{\xymatrix{ 
0\ar[r]& 0 \ar[r]\ar[d]^{0}\ar@{}[rd]|{(\ast)} &\apl(W) \ar@{=}[r]\ar@{=}[d] &\apl(W)\ar[r]\ar[d]_{\apl(j)}& 0 \\
0\ar[r]& \apl(W,\wst) \ar[r]\ar[d]^{\cong \text{ exc}} &\apl(W) \ar[r]_{\apl(j)}\ar[dd]_{\simeq}^{\Phi_W} &\apl(\wst)\ar[r]& 0 \\
 &\apl(T,\delmt) \ar[d]^{\simeq \tilde\Phi_T}& \\
&s^{-n}\#\apl(T,\delpt)\ar[r]&s^{-n}\#\apl(W,\delw)\ar@{=}[d]\\
&s^{-n}\#\apl(K,\delk) \ar[u]_{\simeq} \ar[r]_{s^{-n}\#\apl(i)}&
s^{-n}\#\apl(W,\delw)
}}
\end{equation}
and the two top lines are short exact sequences.

Properties of mapping cones and of short exact sequences imply that,
 in the category of $\apl(W)$-dgmodules,
the morphism
\begin{equation}
\label{E:aplj}
\apl(j)\colon\apl(W)\to\apl(\wst)
\end{equation}
on the top right of \eqref{D:big} is equivalent to the map induced between the mapping cones of the
horizontal maps of the 
square $(\ast)$ in Diagram~$(\ref{D:big})$,
\begin{equation}
\label{E:aplMC}
\id_{\apl(W)}\oplus s0\colon \apl(W)\oplus s0\to\apl(W)\oplus s\apl(W,\wst).
\end{equation}
Since the vertical maps below the second line of \eqref{D:big}
are quasi-isomorphisms, the morphism 
$\id_{\apl(W)}$ $\oplus s0$ in \eqref{E:aplMC}  is equivalent to
\[\iota\circ\Phi_W\colon\apl(W)\to C(s^{-n}\# \apl(i)).\]
The morphism $\apl(j)$ of \eqref{E:aplj} is clearly equivalent to
\[
\apl(W)\to\apl(\wsk).
\]
This finishes the proof.
 \end{proof}

  \bigskip

\section{Rational model of the complement of a subpolyhedron in a manifold with boundary}
\label{s:complement}

\bigskip

In this section we establish the CDGA model of the complement
$W\setminus K$ under some unknotting
condition, in particular when the codimension  of the subpolyhedron is
high  (Theorem \ref{T:mod_conndim}). We also state a partial CDGA model without unknotting condition
(Proposition \ref{modele_truncation}.)


Consider the same setting as at the beginning of Section \ref{s:lef_dual}, in particular Diagram \eqref{diag1}. Suppose given a commutative diagram of CDGAs  
\begin{equation}
\label{diag3}
\xymatrix{
A \ar[d]_{\alpha} \ar[r]^{\varphi} & B \ar[d]^{\beta} \\
\dela  \ar[r]_{\partial \varphi} &\delb}
\end{equation}
that is a CDGA model of  
\begin{equation}
\label{diag4}
\xymatrix{
W &\ar@{_{(}->}_{f}[l] K   \\
\delw  \ar@{_{(}->}[u]   & \ar@{_{(}->}^{\partial f}[l] \delk,\ar@{_{(}->}[u], }
\end{equation}
in other words Diagram \eqref{diag3} is quasi-isomorphic to Diagram ($\ref{diag2}$). 
Note that in Diagram (\ref{diag3}), $\partial A$ and $\partial B$ are just the names of some CDGAs.

The goal of this section is to construct from Diagram \eqref{diag3} a
CDGA model of $\apl(\wsk)$.

\subsection{Dgmodule model of the complement $\wsk$}

Let $\hat{A}$ be a CDGA such that we have the following zig-zag of quasi-isomorphisms 
\begin{equation}
\label{rho_rho'}
\xymatrix{ A&\ar[l]_{\rho}^{\simeq} \hat{A} \ar[r]^-{\rho'}_-{\simeq}& \apl(W). }
\end{equation}
The morphism $\rho'$ induces a structure of $\hat{A}$-dgmodule on
Diagram (\ref{diag2}) and the morphism $\rho$ induces a structure of
$\hat{A}$-dgmodule on Diagram (\ref{diag3}). From Diagram
(\ref{diag3}) we deduce an $\hat{A}$-dgmodules morphism between the
homotopy kernels of $\alpha$ and $\beta$
(see Section \ref{S:hoker})
\begin{equation}\label{barfi}\bar{\varphi}: \text{hoker }\alpha \to \text{hoker } \beta.
\end{equation}
Note also that by Poincaré duality of the pair $(W,\partial W)$, we have a
quasi-isomorphism of $A$-dgmodules
\[\theta_A\colon A\quism s^{-n}\#\hoker\alpha.\]

\begin{prop}
\label{dual}
An $\hat A$-dgmodule model of
\[\apl(W)\to\apl(\wsk)\]
is given by the composite
\[
\xymatrix{
A\ar[r]^-{\simeq}_-{\theta_A}&s^{-n}\#\hoker\alpha\ar@{^(->}[r]_-{\iota}&C(s^{-n}\#\bar{\varphi}) }
\]
where $C(s^{-n}\#\bar{\varphi}) $ is the mapping cone of the
$\hat{A}$-dgmodules morphism
\[s^{-n}\#\bar{\varphi}\colon s^{-n}\#\hoker \beta \to s^{-n}\# \hoker \alpha.\]
\end{prop}


\begin{proof}
Since \eqref{diag3} is a CDGA model of \eqref{diag2}, $\hoker \alpha$ is weakly equivalent as an $\ahat$-dgmodule to $\apl(W,\delw)$ and $\hoker \beta$ is weakly equivalent as an $\ahat$-dgmodule to $\apl(K,\delk)$. Hence, the result is a direct consequence of Proposition \ref{modele_apl_bord}.
\end{proof}

\begin{rem}
If the morphisms $\alpha$ and $\beta$ are surjective then we can work with the genuine kernel instead of the homotopy kernel.
\end{rem}

The major flaw of the dgmodule model of $W\setminus K$ of
Proposition \ref{dual} is that there is no natural CDGA structure on
it. The next proposition is a first step to endow this dgmodule 
 model of $W\setminus K$ with the  structure of a CDGA.
%
\begin{prop}
\label{modele_de_dgmod}
Assume given an $\ahat$-dgmodule morphism $\varphi^! \colon Q \to A$ weakly equivalent to  
\[\sn\bar\varphi \colon \sn\hoker \beta \to \sn \hoker \alpha.\] Then 
an $\hat A$-dgmodule of $\apl(W)\to\apl(\wsk)$ is given by
\[A\,\hookrightarrow\,C(\varphi^!)\]
where $C(\varphi^!)$ is the mapping cone $A\oplus_{\varphi^!} sQ$.
\end{prop}
\begin{proof}
This is a direct consequence of Proposition \ref{dual}.
\end{proof}

\begin{rem} The existence of such a morphism $\varphi^!$ is guaranteed if we take for $Q$ a cofibrant $\hat{A}$-dgmodule model of $s^{-n}\#$ hoker $\beta$.
\end{rem}
%

This new dgmodule model $C(\varphi^!) =$ $A\oplus_{\varphi^!} sQ$ of $
\wsk$ has the advantage that $A$ is a CDGA and therefore,  under some
dimension hypotheses, the semi-trivial CGA structure on the mapping
cone described in Section \ref{s:semi_trivial} makes it into a CDGA.   We develop this in the next section.

%


\subsection{CDGA model of the complement $\wsk$}


We work in the set-up of diagrams
\eqref{diag3}-\eqref{diag4}. Remember also the notion of semi-trivial
CDGA structure on a mapping cone from Section \ref{s:semi_trivial} and
the notion of  CDGA truncation from Section \ref{s:truncation_cdga}. 
Under some codimension and connectedness hypothesis for the inclusion $f\colon K \hookrightarrow W$ we can construct a CDGA  model of $\wsk$.    More   precisely,  we have:

\begin{theo}
\label{T:mod_conndim}
Let $W$ be a compact connected oriented triangulated manifold of
dimension $n$ with boundary, and let $K\subset W$ be
a subpolyhedron of dimension $k$.
Consider Diagram \eqref{diag4} and its CDGA model \eqref{diag3}.
Let $r$ be an integer such that the induced morphisms on homology
$H_*(f;\bq)$ and $H_*(\partial f; \bq)$ are $r$-connected, that is
$H_{\leq r}(W,K;\bq)=0$ and $H_{\leq r}(\partial W,\partial_W K;\bq)=0$. 

Suppose given an $A$-dgmodule $Q$ weakly equivalent to $\sn
\mathrm{hoker} \  \beta$ such that $Q^{<n-k}=0$  and an $A$-dgmodules morphism
\begin{equation}
\label{e:shriek_map}
\varphi^! \colon Q \to A
\end{equation} weakly equivalent to  \[\sn\bar{\varphi} \colon \sn
\mathrm{hoker} \  \beta \to \sn  \mathrm{hoker} \  \alpha.\]

 If
\begin{equation} 
\label{e:connectivity}
r\geq 2k-n+2,
\end{equation}
then every truncation $\tau^{\leq n-r-1} (C(\varphi^!))$ of the
mapping cone $C(\varphi^!)=A\oplus_{\varphi^!} sQ$ equipped with the semi trivial structure is a CDGA, and the morphism 
\[\xymatrix{A\ar[r] & \tau^{\leq n-r-1}(C(\varphi^!))}\]
is a CDGA model of the inclusion
\[\wsk \hookrightarrow W.\]

Moreover it is always possible to construct an $A$-dgmodule $Q$ and a morphism $\varphi^!$ as in \eqref{e:shriek_map}.
\end{theo}

This generalizes the main result of \cite[Theorem
1.2]{LaSt:PE} to manifolds with boundary.
 A first direct  consequence of this theorem is the following
 corollary on the rational
 homotopy invariance of the complement under some
 connectedness-codimension hypotheses.
\begin{cor}
\label{C:mod_conndim}
Let $W$ be a compact triangulated manifold with boundary and $K\subset W$ be a subpolyhedron. Assume that $W$ and $\delw$ are 1-connected and that the inclusions 
\[K \hookrightarrow W \text{     and    } K\cap  \delw \hookrightarrow \delw \]
are $r$-connected with
\begin{equation}
\label{unknot2}
r \geq 2 (\dim K) - \dim W +2.
\end{equation}
Then the rational homotopy type of $\wsk$ depends only on the rational homotopy type of the diagram 
\begin{equation*}
 \xymatrix{
\delk \ar@{^{(}->}^{\partial f}[r] \ar[d]   & \delw\ar[d]  \\
K \ar@{^{(}->}_{f}[r] &W. } 
\end{equation*} 
\end{cor}

The hypotheses \eqref{e:connectivity} (or equivalently \eqref{unknot2})  is called \emph{the unknotting condition} and it cannot be removed as shown in \cite[Section 9]{LaSt:PE}.

 \begin{proof}[Proof of Theorem \ref{T:mod_conndim}]
Let  $\ahat$ be a CDGA such that we have a zig-zag of  CDGA quasi-isomorphisms 
\[\xymatrix{\apl(W) &\ar[l]^-{\rho'}_-{\simeq} \hat{A} \ar[r]_-{\rho}^-{\simeq} &A.} \]
Set $N \defeq 2(n-k)-3$.   By  Proposition \ref{p:truncMC} (with
$p=n-k$), 
$\troncon$ admits the structure of a CDGA induced by the semi-trivial
CGA structure on the mapping cone, and
the composite
\[\xymatrix{ A \ar[r]^{\iota \ \ \ } & C(\varphi^!) \ar[r] & \troncon}\]
is a CDGA morphism.

We now prove that $H^{>N}(W\setminus K)=0$ where (co)homology of spaces
is understood with coefficients in $\BQ$. By excision and the
connectedness hypotheses on $H(\partial f)$ and $H(f)$,
\[H_{\leq r}(K\cup_{\partial_W K}\partial W\,,\,K)\,\cong\,H_{\leq
r}(\partial W\,,\,\partial_WK)\,=\,0\]
and
\[H_{\leq
r}(W\,,\,K)\,=\,0.\]
Lefschetz duality and the long exact sequence of the triple\\
$(W\,,\,K\cup_{\partial_WK}\partial W\,,\,K)$ give
\[
H^{\geq n-r}(W\setminus K)\,\cong\,H_{\leq
  r}(W\,,\,K\cup_{\partial_WK}\partial W)\,=\,0.\]
The unknotting hypothesis ($\ref{e:connectivity}$) implies that
$N\geq n-r-1$, therefore $H^{>N}(W\setminus K)=0$.

By Proposition \ref{modele_de_dgmod}, $A\to C(\varphi^!)$ is an $\ahat$-dgmodule model of $\apl(W)\to\apl(\wsk)$.  This implies that
\[
H^{>N} (C(\varphi^!)) \cong H^{>N} (\wsk)=0,
\]
therefore
\[
\operatorname{proj}: C(\varphi^!) \longrightarrow \tau^{\leq N} (C(\varphi^!))
\]
is a quasi-isomorphism.

Thus the CDGA morphism
\[A\to \tau^{\leq N}\left(C(\varphi^!)\right)\]
 is a model of $\hat A$-dgmodules of
$\apl(W)\to\apl(\wsk)$. We will prove that it is actually a CDGA model.

Take a minimal relative Sullivan model (in the sense of
\cite[Chapter 14]{FHT:RHT})
\begin{equation}
\label{E:mimodel}
\xymatrix{\hat{A} \ar@{->>}[r]^{\rho'}_{\simeq} \ar@{ >->}[dr] & \apl(W)\ar@{->>}[r] & \apl(\wsk) \\
&(\hat{A}\otimes \Lambda V,D). \ar@{->>}[ur]^{\lambda'}_{\simeq}}
\end{equation}
By Proposition \ref{modele_de_dgmod}, $\xymatrix{ \ahat \ar@{ >->}[r]
  & \ahat\otimes \Lambda V}$ is an $\ahat$-dgmodule model of $A\to
C(\varphi^!)$.    Since $(\hat A \otimes \Lambda V, D)$ is a cofibrant
$\hat A$-dgmodule,  we can construct a weak equivalence of
$\ahat$-dgmodules
\[\lambda\colon  \ahat\otimes \Lambda V\to C(\varphi^!)\]
making commute 
 the  following diagram, where the upper
part is 
of CDGA and the lower part is of $\hat A$-dgmodules, 
\begin{equation}
\label{D:diagmodel}
\xymatrix{\apl(W)\ar[r]& \apl(\wsk)& \\
\ahat  \ar[d]^{\simeq}_{\rho} \ar[u]_{\simeq}^{\rho'} \ar@{>->}[r] &
\ahat\otimes \Lambda V
\ar[dr]_-{\simeq}^-{\bar\lambda=\proj\circ\lambda}\ar[u]_{\simeq}^{\lambda'}
\ar@{-->}[d]^{\simeq}_{\lambda}
& \\
 A \ar[r] & C( \varphi^!)  \ar[r]^-{\simeq}_-{\proj} & \troncon. } 
\end{equation}

By Lefschetz duality and the hypothesis on the dimension of $K$
\[
H^{<n-k} (W,\wsk)   \cong H_{>k} (K,\partial_W K)=0.
\]
By minimality of the Sullivan relative model
\eqref{E:mimodel}, this implies that $V^{<n-k-1}=0$.     Therefore $(\Lambda^{\geq 2}V)^{\leq N}=0$ and, since $(\tau^{\leq N}C(\varphi^!))^{>N}=0$, this implies that the composition \label{rho} 
\[
\bar\lambda :\ahatv \stackrel{\lambda}{\longrightarrow}C(\varphi^!) \stackrel{\textrm{proj}}{\longrightarrow} \tau^{\leq N} (C(\varphi^!))
\] 
is a morphism of CDGA. Thus all the solid arrows in Diagram
$(\ref{D:diagmodel})$ are of CDGAs.  This achieves to prove that $A\to
\tau^{\leq N} (C(\varphi^!))$ is a CDGA model of  $\wsk
\hookrightarrow W$, as claimed.

It remains to prove the existence of an $A$-dgmodule $Q$ and a
morphism $\varphi^!$.  Since $H^{>k}$(hoker $\beta$) $\cong H^{>k}(K,\partial_W K)=0$, we have $H^{<n-k} (s^{-n} \#$ hoker $\beta)=0$.    Therefore there exists a cofibrant $A$-dgmodule model $Q$ of $s^{-n} \#$ hoker $\beta$ such that $Q^{<n-k}=0$.    Since, by Poincar\'e duality, $s^{-n}\#$ hoker $\alpha \simeq A$, there exists an $A$-dgmodule morphism
\[
\varphi^! : Q \longrightarrow A
\]
weakly equivalent to $s^{-n}\# \bar \varphi$.
\end{proof}

Actually even when the unknotting condition \eqref{e:connectivity} of
Theorem \ref{T:mod_conndim} is not satisfied, we still   get a partial
model of $\wsk$. More precisely we get a CDGA model of $\wsk$ up to
some degree,  i.e. a model of the truncation of $\apl(\wsk)$.  This is
the content of the next proposition.

\begin{prop}
\label{modele_truncation}
Consider the same hypotheses as in Theorem \ref{T:mod_conndim} except that we do not assume the unknotting condition \eqref{e:connectivity}

Let $l \colon A(W) \to A (\wsk)$ be a CDGA model of $\apl(W) \to \apl(\wsk)$ such that $A(W)$ and $A(\wsk)$ are connected. Set $N=2(n-k)-3$. Then the CDGA morphism 
\[  A  \to \troncon \]
is a CDGA model of the composite  
\[\pi \circ l \colon A(W) \hookrightarrow A(\wsk) \to \tau^{\leq N} A(\wsk). \]

\end{prop}

\begin{proof} [Proof of Proposition \ref{modele_truncation}] The proof
  is very similar to that of Theorem \ref{T:mod_conndim}. The details
  to change are left to the reader. 
\end{proof}

\begin{rem}
We would have preferred in Proposition \ref{modele_truncation} to
state that $A\to \tau^{\leq N} (C(\varphi^!))$ is a CDGA model of
$\apl{(W)}\to \tau^{\leq N} (\apl(\wsk))$.    But the latter is not
well defined because $\apl(\wsk)$ is not connected and hence we cannot
take its truncation.   This is the reason for consi\-dering instead a
model 
$l\colon A(W)\to A(\wsk)$ between \emph{connected} CDGAs.
\end{rem}

Note that $N=n-(2k-n+2)-1$ and therefore, under the unknotting condition $r \geq 2k-n+2$, we have that $N\geq n-r-1$. But,  Poincar\'e duality and the $r$-connectedness implies that $H^{\geq n-r} (\wsk)=0$. Hence Theorem \ref{T:mod_conndim} is actually a corollary of Proposition \ref{modele_truncation}

\section{Rational model of the configuration space of two points in a manifold with boundary}
\label{S:ConfW2}
\medskip

In this section we use the results of Section \ref{s:complement} to describe the rational homotopy type of the configuration space of two points in a compact manifold with boundary under 2-connectedness hypotheses.    
In particular we prove in Corollary \ref{c:rht_fw2} that the rational
homotopy type of $\Conf(W,2)$ depends only on the rational homotopy
type of the pair $(W,\delw)$ when $W$ and $\delw$ are 2-connected. We
also construct in Theorem \ref{t:modele_adgc_FW2_1} an explicit CDGA
model of $\Conf(W,2)$. Moreover in Theorem \ref{T:modele_FW2} we
describe an elegant CDGA model for $\Conf(W,2)$ when the pair
$(W,\partial W)$ admits a pretty surjective model in the sense of
\cite{CLS:pretty}.

Fix a compact connected orientable manifold of dimension $n$, $W$, with boundary $\delw$.  
Let \[\Delta : W \hookrightarrow W\times W \ ; x\mapsto (x,x)\]
be the diagonal embedding. The configuration space of two points in $W$ is the complementary space
\[\Conf(W,2):= (W\times W) \backslash \Delta(W)=\{ (x,y) \in W\times W | x\neq y\}.\]
 Notice that the diagonal embedding $\Delta$ is such that
 $\Delta(\partial W) \cong \partial W$ and $\Delta^{-1} (\delw \times
 \delw) = \delw$.     In other words, with the notation of \eqref{E:dWK}, 
 $$
 \partial_{W\times W}\left(\Delta(W)\right)=\Delta(\partial W) \cong \partial W.
 $$ 
   Therefore, according to Corollary \ref{C:mod_conndim}, if $W$ and
   $\partial W$ are connected enough,  then the rational homotopy type
   of $\Conf(W,2)=W\times W \backslash \Delta (W)$ is determined by
   the square $(\ref{diagFW2_2})$ of Proposition
   $\ref{modele_carre_FW2}$ below.   The goal of the next section is to compute a CDGA model of that square.

\subsection{CDGA model of the diagonal embedding of the pair $(W,\delw)$ into $(W\times W, \partial (W\times  W))$}

The goal of this section is to prove the following proposition.

\begin{prop}
\label{modele_carre_FW2}
Let $W$ be a compact connected orientable manifold with \  boundary  \ $\delw$.  \ Suppose  \ given  \ a CDGA surjective model $\xymatrix{\beta \colon B \ar@{->>}[r]& \delb}$ of the inclusion  $\delw \hookrightarrow W$. Then a CDGA model of the square
\begin{equation}
\label{diagFW2_2}
\xymatrix{W\times W && \ar@{_{(}->}[ll]_-{\Delta} W  \\
  \partial(W\times W) \ar@{_{(}->}[u]& &\ar@{_{(}->}[ll]^-{\partial\Delta}  \delw. \ar@{_{(}->}[u]}
\end{equation}
where $\Delta$ is the diagonal map and $\partial \Delta$ is the composition $\delw \stackrel{\Delta}{\hookrightarrow} \delw \times \delw \hookrightarrow \partial (W \times W)$
is given by the CDGA square
\begin{equation}
\label{diagFW2ADGC_2}\xymatrix{\btb \ar[r]^{\mu}\ar[d]_{\alpha} & B \ar[d]^{\beta}\\
\frac{\btb}{(\ker\beta \otimes \ker \beta)} \ar[r]_-{\tilde \mu} & \delb }
\end{equation}
where $\mu$ is the multiplication, $\alpha$ is the projection on the quotient, and $\tilde{\mu}$ the map induced by $\beta \circ \mu$.

\end{prop}

The rest of this section is devoted to the proof of this result and for the rest of it we will use the notations introduced in the proposition.
First, notice that, since $W$ is a manifold with boundary, $W\times W $ is also a manifold with boundary \[\partial(W\times W)= W\times \delw \cup_{\delw\times\delw} \delw \times W.\]
In other words we have a pushout (and homotopy pushout)  
\begin{equation}
\label{d:bound}
\xymatrix{
\partial(W\times W) \ar@{}[rd]|{\textrm{pushout}}&\ar@{_{(}->}[l] W \times (\partial W) \\
(\partial W)\times W\ar@{^{(}->}[u] & \ar@{_{(}->}[l]\ar@{^{(}->}[u] \partial W \times \partial W.}
\end{equation} 
 

The key argument to prove Proposition \ref{modele_carre_FW2} is  that
Diagram \eqref{diagFW2_2} is the right upper half of the following diagram
\begin{equation}
\label{D2}
\xymatrix{ W\times W &&\ar@{_{(}->}[ll]_{\Delta} W \\
 \ar@{^{(}->}[u]\partial(W\times W) \ar@{}[rd]|{\textrm{pushout}}  & \ar@{_{(}->}[l]W\times (\partial W) \\
(\partial W) \times W \ar@{^{(}->}[u] & \ar@{_{(}->}[l]\delw \times \delw \ar@{^{(}->}[u] & \ar@{_{(}->}[l] \partial W\ar@{^{(}->}[uu]} 
\end{equation}
where the maps are the obvious inclusions and diagonals, and the small
left lower square in \eqref{D2} is the homotopy pushout \eqref{d:bound}.

\begin{lem}
\label{lem_po4}
The following diagram is a CDGA model of diagram \eqref{D2} 
\begin{equation}
\label{po4}
\xymatrix{B\otimes B \ar[d]^{\alpha}\ar[rr]^-{\mu}&& B\ar@{->>}[dd]_-{\beta}\\
P \pullback\ar[r]\ar[d] & B\otimes \delb\ar@{->>}[d]^{\beta\otimes \oid} &\\
\delb\otimes B  \ar@{->>}[r]_{\oid \otimes \beta} &\delb \otimes \delb \ar[r]^{\mu}& \delb
}
\end{equation}
where $P$ is the pullback of the small square, $\alpha$ is the morphism given by the universal property, and $\mu$ are the multiplication morphisms.
\end{lem}
\begin{proof}[Proof   of Lemma \ref{lem_po4}]  

Using the classical CDGA models for products and diagonal maps on spaces, the fact
that $\apl$ turns homotopy pushout of topological spaces into homotopy
pullbacks of CDGAs, that a pullback of CDGA surjections is a homotopy
pullback, and standard techniques in rational homotopy theory
we get that a CDGA model of Diagram \eqref{D2} is given by the following diagram, where $P'$ denotes the pullback of the left bottom corner of the square,
\begin{equation}
\label{po3}
\xymatrix{\apl(W)\otimes \apl(W) \ar[d]\ar[rr]^{mult} && \apl(W)\ar[dd] \\
P' \pullback \ar[r] \ar[d]& \apl(W)\otimes \apl(\delw)\ar@{->>}[d] & \\
\apl(\delw) \otimes \apl(W) \ar@{->>}[r]&\apl(\delw) \otimes \apl(\delw) \ar[r]& \apl(\delw).}
\end{equation}
This diagram is easily seen to be equivalent to Diagram \eqref{po4}.



\PL{Since $\xymatrix{\beta \colon B \ar@{->>}[r]& \delb}$ is a CDGA model of the morphism  $\delw \hookrightarrow W$, it is easy to construct a surjective CDGA morphism $\xymatrix{\hat \beta \colon\hat B \ar@{->>}[r]& \widehat \delb}$ such that the diagram below commutes
\begin{equation}
\label{bhat}
\xymatrix{B\ar@{->>}[d]_\beta & \ar[l]\hat{B} \ar[r]\ar@{->>}[d]_{\hat{\beta}}& \apl(W)\ar@{->>}[d] \\
\delb &\ar[l] \widehat\delb \ar[r]& \apl(\delw).}
\end{equation}

%

Let us show that the diagram
\begin{equation}
\label{po5}
\xymatrix{
\hat B\otimes \hat B \ar[d]^{\hat \alpha}\ar[rr]^-{\mu}&& \hat B\ar@{->>}[dd]_-{\hat\beta}
 \\
\hat{PB} \ar[r]\ar[d] & \hat B\otimes \hat \delb\ar@{->>}[d]_{\hat\beta\otimes \oid} &\\
\widehat\delb\otimes \hat{B} \ar@{->>}[r]^{\oid \otimes \hat{\beta}} &\widehat\delb \otimes \widehat\delb \ar[r]^-{\mu}& \widehat\delb,}
\end{equation}
where $\hat PB$ is the pullback of the upper left square, $\hat\alpha$ is the morphism given by the universal property and $\mu$ describes the multiplication morphism, is a CDGA model of Diagram \eqref{po3}. 

For this, look at the following quasi-isomorphisms diagram 
{\small\begin{equation}
\label{qi_po2}
\xymatrix{\hat \delb\otimes \hat B\ar[d]_{\simeq} \ar@{->>}[r]& \widehat\delb \otimes \widehat\delb\ar[d]_{\simeq} & \ar@{->>}[l]\hat B\otimes \widehat\delb \ar[d]^{\simeq} \\
\apl(\delw)\otimes \apl(W) \ar@{->>}[r]& \apl(\delw) \otimes \apl(\delw) & \ar@{->>}[l]\apl(W)\otimes \apl(\delw). }\end{equation}}

By universal property of the Pullback $P'$ described in Diagram \eqref{po3}, there exists a unique CDGA morphism 
\[\hat{\rho} \colon \hat{PB} \to P'\]
which is, since all the vertical arrows in \eqref{qi_po2} are quasi-isomorphisms, a quasi-isomorphism.  

On the other hand, we have, once again by the universal property of the pullback P', a unique morphism 
\[v \colon \apl(W)\otimes \apl(W) \to P'\]
and by universal property of $\hat {PB}$ 
\[\hat{\alpha} \colon \hat B \otimes \hat B \to \hat{PB}.\]
 We easily verify that the morphisms
\[\xymatrix{\hat B\otimes \hat B \ar[r]^{\hat{\alpha}} & \hat {PB} \ar[r]^{\hat{\rho}} & P'} \]
and
\[\xymatrix{\hat B\otimes \hat B \ar[r] & \apl(W)\otimes \apl(W) \ar[r]^-v & P'}\]
verifies the universal property. Since this morphism is unique, the following diagram commutes 
\[\xymatrix{\hat B\otimes \hat B \ar[r]^-{\hat \alpha}\ar[d]_{\simeq} & \hat PB \ar[d] ^{\simeq} \\
\apl(W)\otimes \apl(W) \ar[r]^-{\upsilon} & P'.}\]
Moreover, the fact that  $\mu \colon \hat{B} \otimes \hat B \to \hat B $ is a CDGA model of the morphism $\mu\colon \apl(W)\otimes \apl(W) \to \apl(W)$ finalizes to prove that $\eqref{po5}$ is a CDGA model of the diagram \eqref{po3}.
An analogous reasonment shows that \eqref{po5} is a CDGA model of
\eqref{po4}.}

\end{proof}

The following lemma computes the small lower  left pullback square in Diagram $(\ref{po4})$.

\begin{lem}
\label{l:PBCDGA}
We have a pullback in CDGA 
\begin{equation}
\label{PBCDGA}\xymatrix{\frac{\btb}{(\ker\beta \otimes \ker \beta)}
  \ar[r]^{\overline{id_B \otimes \beta}}\ar[d]_{\overline{\beta\otimes
      id_B}} &
 B\otimes \delb \ar[d]^{\beta\otimes id_{\partial  B}}\\
 \delb \otimes B\ar[r]_-{id_{\partial B}\otimes \beta} & \delb\otimes \delb. }
\end{equation}
\end{lem}
\begin{proof}
Consider the following diagram of CDGA's where the internal square is a pullback and $\alpha$ is the map induced by the universal property: 
\[
\xymatrix{ \btb \ar@{-->}[rd]^{\alpha} \ar[rrd]^{\beta \otimes id_B} \ar[rdd]_{id_\beta\otimes \beta} \\
& P\ar[r]\ar[d] &  B\otimes \partial B\ar@{->>}[d]^{\beta \otimes id_{\partial B}} \\
&\delb\otimes B \ar@{->>}[r]_{ id_{\partial B} \otimes \beta } & \delb\otimes \delb. }\]
It is straightforward to check that $\alpha$ is surjective and that 
\[\ker\alpha = \ker \beta \otimes \ker \beta.\]
Therefore we have an induced isomorphism 
\[\overline{\alpha} \colon \frac{B\otimes B}{\ker \beta \otimes \ker \beta}  \stackrel{\cong}{\longrightarrow} P.\]

\end{proof}

\begin{proof}[Proof of Proposition \ref{modele_carre_FW2}]
Diagram  (\ref{diagFW2_2}) is the upper right part of Diagram (\ref{D2}), therefore, by Lemma  \ref{lem_po4}, a CDGA model of (\ref{diagFW2_2})  is given by the upper right part of (\ref{po4}).   Using Lemma \ref{l:PBCDGA} which computes the pullback $P$, we deduce that this CDGA model is $(\ref{diagFW2ADGC_2})$.
\end{proof}

\subsection{A first CDGA model of $\Conf(W,2)$}

Let $\betab$ be a surjective CDGA model of $i \colon\delw \hookrightarrow W$. Using the results of Section~\ref{s:complement}, a CDGA model  of $\Conf(W,2)=W\times W\backslash \Delta(W)$ can be obtained from a CDGA model of 
\begin{equation}
\label{diagFW2_3}
\xymatrix{W\times W && \ar@{_{(}->}[ll]_-{\Delta} W  \\
  \partial(W\times W) \ar@{_{(}->}[u]& &\ar@{_{(}->}[ll]^-{\partial\Delta}  \delw= \partial_{W\times W} W. \ar@{_{(}->}[u]}
\end{equation}
which, by Proposition \ref{modele_carre_FW2} is given by 

\begin{equation}
\xymatrix{\btb \ar[r]^{\mu}\ar[d]_{\alpha} & B \ar[d]^{\beta}\\
\frac{\btb}{(\ker\beta \otimes \ker \beta)} \ar[r]_-{\tilde \mu} & \delb. }
\end{equation}

\begin{theo}
\label{t:modele_adgc_FW2_1}
Let $W$ be a compact triangulated manifold with boun\-dary such that $W$ and $\delw$ are $2$-connected.
Let $\betab$ be a surjective CDGA model of $\delw \hookrightarrow W$  and consider the map 
\[\bar \mu \colon \ker \beta \otimes \ker \beta \to \ker \beta\]
induced by the multiplication $\mu \colon B\otimes B \to B$.   Suppose given a $\btb$-dgmodule morphism 
\[\delta^! \colon D \to \btb \]
weakly equivalent to 
\[\sdn \bar\mu \colon \sdn \ker \beta \to \sdn (\ker \beta \otimes \ker \beta)\]
and such that $D^{<n}=0$. 

Then every truncation $\tau^{\leq 2n-3}C(\delta^!)$ of the mapping
cone of $\delta^!$ admits a semi-trivial CDGA structure  and  
\[\btb \to \tau^{\leq 2n-3}C(\delta^!)\]
is a CDGA model of $\apl(W\times W) \to \apl(\Conf(W,2))$.
\end{theo}
\begin{proof}
Since $W$ and $\delw$ are 2-connected, we have that the morphisms $\Delta \colon W\hookrightarrow W \times W$ and $\partial \Delta \colon \delw \hookrightarrow \partial (W\times W)$ are $2$-connected. So we are under the hypothesis of Theorem \ref{T:mod_conndim} with $r=2$, and the result is a direct consequence of it.
\end{proof}

We deduce the rational homotopy invariance of $\Conf (W,2)$. 

\begin{cor}
\label{c:rht_fw2}
Let $W$ be  a compact manifold with boundary.    If $W$ and $\delw$ are $2$-connected then the rational homotopy type of $\Conf(W,2)$ depends only of the rational homotopy type of the pair $(W, \delw)$. 
\end{cor}
The rational homotopy invariance of Conf $(W,2)$ when $W$ is a closed 2-connected has been established in \cite{LaSt:FM2}, and \cite{Cor:FM2-1conn} gives partial results in the 1-connected case.      When $W$ is   not simply-connected, \cite{LoSa:con} shows that there is no rational homotopy invariance.

\begin{rem}
\label{R:BB'}
If we have a CDGA quasi-isomorphism $B \qi B'$ and $\delta'^! \colon
D' \to B'\otimes B'$ a $B'\otimes B'$-dgmodule  
morphism which is weakly equivalent as a $\btb-$dgmodule morphism to
$\sdn \bar\mu$ then
it follows immediately from Theorem \ref{t:modele_adgc_FW2_1} that
 \[B'\otimes B' \to \tau^{\leq 2n-3}C(\delta'^!)\] is also a CDGA model of $\apl(W\times W) \to \apl(\Conf(W,2))$. 
\end{rem}

\subsection{A CDGA model of $\Conf(W, 2)$ when $(W,\delw)$ admits a surjective pretty model}\label{S:ConfPretty}

Let $W$ be a compact manifold of dimension $n$ with boundary $\delw$
such that both $W$ and $\delw$ are 2-connected. 
In this section we will construct an elegant CDGA model of
$\Conf(W,2)$ when the pair $(W,\partial W)$ admits a \emph{surjective
  pretty model} in the sense of \cite[Definition 3.1]{CLS:pretty}. Let us
recall what this means.
Suppose given 
\begin{itemize}
\item[(i)] a connected Poincaré duality CDGA, $P$, in dimension $n$ ;
\item[(ii)] a connected CDGA, $Q$; 
\item[(iii)] a CDGA morphism, $\varphi \colon P \to Q$.
\end{itemize}
Since $P$ is a Poincaré Duality CDGA there exists an isomorphism of $P$-dgmodules 
\begin{equation}
\label{Eq:thetaP} 
\theta_P \colon P \stackrel{\cong}{\longrightarrow} \sn P.
\end{equation}
Consider the composite 
\begin{equation}
\label{eq:phishriek} 
\xymatrix{\varphi^! \colon \sn Q \ar[r]^{\sn \varphi} & \sn P \ar[r]^{\theta_P^{-1}}& P,}
\end{equation}
which is a morphism of $P$-dgmodules. Assume that the morphism 
\[\varphi \varphi^! \colon \sn Q \to Q\]
is balanced (see Definition \ref{D:balanced}) and  consider the CDGA morphism
\begin{equation}
\label{Eq:wdf} 
\varphi \oplus \oid \colon P \oplus_{\varphi^!} s\sn Q \to Q \oplus_{\varphi \varphi^!} s\sn Q.
\end{equation}
When \eqref{Eq:wdf}   is a CDGA  model of the inclusion $\delw
\hookrightarrow W$ we say that it is a \emph{pretty model} of the pair
$(W,\partial W)$. If moreover $\varphi$
is surjective (and hence also \eqref{Eq:wdf}) we say that is is a
surjective pretty model. 
Then if we consider the differential ideal 
\begin{equation}
\label{Eq:ideal}
I= \varphi^! (\sn Q) \subset P, 
\end{equation}
\cite[Corollary 3.3]{CLS:pretty}  states that the CDGA $P/I$ is a CDGA model of  $W$. 
In \cite{CLS:pretty} we proved that many compact manifolds admit
surjective pretty models as for examples even-dimensional disk bundles
over closed manifolds, complements of high codimensional polyhedra
in a closed manifold, as well as any compact manifold whose boundary
retracts rationally on its half-skeleton (see \cite[Definition 6.1]{CLS:pretty}.)

The objective in this section is to use this  model, $P/I$, of $W$ to construct an elegant model for $\Conf (W,2)$, analogous to the one constructed in \cite{LaSt:FM2} for configuration spaces in closed manifolds.

Since $P$ is a Poincar\'e duality CDGA, for any homogeneous basis $\{a_i\}_{0\leq i \leq N}$ of $P$, there exists a Poincar\'e dual basis $\{a^\ast_i\}_{0\leq i \leq N}$ charac\-terized by $\epsilon (a_ia^\ast_j)=\delta_{ij}$ where $\epsilon : P^n \to \bq$ is an orientation of $P$ and $\delta_{ij}$ is the Kronecker symbol.      
Let $\Delta \in (P\otimes P)^n$ be the diagonal class of $P\otimes P$ defined as 
\begin{equation}
\label{eq:diag}
\Delta = \sum_{i=0}^N (-1)^{|a_i|} a_i\otimes a_i^*.
\end{equation}
Denote by
\[\pi \colon P \to P/I\]
the projection. Taking the image of the diagonal $\Delta$ by the   projection $\pi \otimes \pi \colon P\otimes P \to P/I \otimes P/I$ we get a \emph{truncated  diagonal class}  
 \begin{equation}
\label{eq:trunc_diag}
\overline {\Delta}= (\pi \otimes \pi) (\Delta)\in (P/I \otimes P/I)^n.
\end{equation}
Define the 
map
\begin{equation}
\label{e:delta_shriek}
\overline{\Delta}^! \colon s^{-n}P/I \to P/I \otimes P/I \ ; \ s^{-n}x \mapsto \overline{\Delta} \cdot (1\otimes x).
\end{equation}
\begin{lem}
\label{l:trunc_diag_bob}
The map $\overline{\Delta}^! \colon s^{-n}P/I \to P/I\otimes P/I$
defined in \eqref{e:delta_shriek}  is a $P/I \otimes P/I$-dgmodules morphism. 
\end{lem}
\begin{proof}
In \cite[Lemma 5.1]{LaSt:FM2} it is shown that for $P$ a connected Poincar\'e duality CDGA, the morphism 
$\Delta^! \colon s^{-n} P \to P\otimes P \ ; \ s^{-n} x \mapsto \Delta (1\otimes x)$ is  a $P\otimes P$-dgmodules morphism. 

We have the following commutative diagram 
\[\xymatrix{ 
s^{-n}P \ar[d]_{s^{-n}\pi}\ar[r]^{\Delta^!} & P\otimes P
\ar[d]^{\pi\otimes \pi}\\
s^{-n}P/I \ar[r]_{\overline{\Delta^!}} & P/I \otimes P/I.
}\]
Since $P/I$ is a $P$-dgmodule generated by $1\in P/I$, this implies that $\overline \Delta ^! $ is a $P\otimes P$-dgmodules morphism, and  the surjectivity of the morphism $\pi:P \to P/I$ implies that $\overline \Delta ^! $ is a $P/I\otimes P/I$-dgmodules morphism. 
\end{proof}

The main result of this section is the following theorem.

\begin{theo}
\label{T:modele_FW2}
Let $W$ be a $2$ connected compact manifold of dimension $n$ whose
boundary is $2$-connected. Suppose that $(W,\delw)$ admits a
surjective  pretty model of the form $(\ref{Eq:wdf})$, and let $\overline{\Delta}^!$ be the $P/I\otimes P/I$-dgmodules morphism defined in \eqref{e:delta_shriek}. Then the mapping cone 
 \[C (\overline{\Delta}^!)=\left( P/ I \otimes P/I \right) \oplus_{\overline{\Delta}^!} ss^{-n}P/I,\] 
   equipped with the semi-trivial structure  is a CDGA model of  $\Conf (W,2)$. 
\end{theo} 

Before proving the theorem, let us fix some notation and prove a lemma.
Set 
\[B = P\oplus_{\varphi^!} s\sn Q \text{\quad and \quad} \delb= Q\oplus s\sn Q.\]
By hypothesis
\begin{equation}
\label{Eq:beta}
\beta \defeq \varphi \oplus \oid \colon B \twoheadrightarrow \delb
\end{equation}
is a surjective CDGA model for the inclusion $\delw \hookrightarrow W$.  
Also let $B':=P/I$ and notice that the obvious projection $\pi\oplus
0\colon B \stackrel{\simeq}{\longrightarrow}B'$ is a quasi-isomorphism of
CDGA.
  According to Theorem \ref{t:modele_adgc_FW2_1} and Remark
  \ref{R:BB'}, we only need to show that
 $\overline{\Delta}^!$ is equivalent to $s^{-2n}\#\bar \mu$, which is the content of the following lemma.

\begin{lem}
\label{lem_PI_mubar}
There exists a $B\otimes B$-dgmodules commutative square:
\begin{equation}
\label{diag_PI_mubar}
\xymatrix{s^{-n}P/I \ar[rr]^{\overline{\Delta}^!}\ar[d]^{\cong}_{\overline{\theta}_P} && P/I\otimes P/I \ar[d]_{\cong}^{\overline{\theta}_{P\otimes P}} \\
\sdn \ker\beta \ar[rr]_-{\sn\bar{\mu}} && \sdn \ker\beta \otimes \ker\beta.}
\end{equation}
\end{lem}
\begin{proof}
 By Poincar\'e duality of the CDGA $P$, we have a $P$-dgmodules
 isomorphism $\theta_P:   P  \stackrel{\cong}{\longrightarrow} \sn P
 $. This morphism induces, by construction of the differential ideal
 $I\subset P$ (see \eqref{eq:phishriek} and \eqref{Eq:ideal}), a $P$-dgmodules isomorphism 
\begin{equation*}
\bth_P \colon P/I \stackrel{\cong}{\longrightarrow} \sn \ker \varphi.
\end{equation*}
The morphism
\begin{equation*}
\beta \defeq \varphi \oplus \oid \colon B \twoheadrightarrow \delb
\end{equation*}
 is a surjective CDGA model of $\delw \hookrightarrow W$. We have an
 obvious isomorphism 
$\ker\beta \cong \ker \varphi$ as $P$-dgmodules. So, we have a $P$-dgmodules isomorphism (that we will also denote $\bth_P$)
\[
\bth_P \colon P/I \stackrel{\cong}{\longrightarrow} \sn \ker \beta.
\]
An easy computation   shows that for $(p,u) \in B=P\oplus s\sn Q$ and $x\in P/I$,
\[
\bth_P ( (p,u) \cdot x) =(p,u) \bth (x).
\]
Thus $\bth_P$ is a morphism of $B$-dgmodules and, via the multiplication $\mu \colon B\otimes B \to B$, it is a $B\otimes B$-dgmodules morphism. As a direct consequence we have the $B\otimes B$-dgmodules isomorphism 
 \[\overline{\theta}_{P\otimes P} \colon P/I \otimes P/I \stackrel{\cong}{\longrightarrow} \sn \ker \beta \otimes \sn\ker\beta \cong \sdn (\ker \beta \otimes \ker\beta) \]

By Lemma \ref{l:trunc_diag_bob}, the morphism $\overline{\Delta}^!$ is
a $P/I \otimes P/I$-dgmodules morphism, 
and hence it is also a morphism of $B\otimes B$-dgmodules.

Consider the following diagram of $B\otimes B$-dgmodules
\[\xymatrix{s^{-n}P/I \ar[rr]^{\overline{\Delta}^!}\ar[d]^{\cong}_{\overline{\theta}_P} && P/I\otimes P/I \ar[d]_{\cong}^{\overline{\theta}_{P}\otimes \bth_P} \\
\sdn \ker\beta \ar[rr]_-{\sn\bar{\mu}} && \sdn \ker\beta \otimes
\ker\beta,}\]
and let us show that it commutes.
Since $P/I$ is a $B\otimes B$-dgmodule generated by the element $1\in
P/I$,
 it suffices to prove that 
\[\bth_P \otimes \bth_P (\overline{\Delta}^! (s^{-n}1)) = \sn \bar\mu (\bth_P(s^{-n}1)).\]
A straightforward computation shows that this is the case. 
\end{proof}



\begin{proof}[Proof of Theorem \ref{T:modele_FW2}]
Since $W$ and $\delw$ are $2$-connected, Lemma \ref{lem_PI_mubar},
Remark \ref{R:BB'} and Theorem \ref{t:modele_adgc_FW2_1} imply that   
 \[
 P/I \otimes P/I \to \tau^{\leq 2n-3}C(\overline{\Delta}^{!})
 \] 
 is a CDGA model of  $\Conf(W,2) \hookrightarrow W\times W$. Moreover, we can verify that the morphism $\overline{\Delta^!}$ is balanced, therefore $C(\overline{\Delta}^!)$ is also a CDGA when equipped with the semi-trivial structure. By the $2$-connectedness of the manifold $W$ and for degree reasons we have that $C(\overline{\Delta}^!)\qi \tau^{\leq 2n-3}C(\overline{\Delta}^{!})$ 
is a CDGA quasi-isomorphism.
\end{proof}


\subsection{A CDGA model for $\Conf(W,2)$ when $W$ is a disk bundle of even rank over a closed manifold}\label{S:ConfDxi}

We apply the model constructed in Section \ref{S:ConfPretty} to disk bundles.

Let $\xi$ be a vector bundle of even rank, $2k$, for some $k\geq 2$,
over some 2-connected closed manifold, $M$, of dimension $m$.  Then
the disk bundle $D\xi$ is a compact manifold of dimension $m+2k$ with boundary the sphere bundle $S\xi$.

Let $Q$ be a Poincar\'e duality CDGA model of $M$, let
\[
\Delta_Q \in (Q\otimes Q)^m
\]
be a diagonal class for $Q$, and let
\[
e\in Q^{2k} \cap \ker (d_Q)
\]
be a representative of the Euler class of $\xi$. 
Denote by $(\Delta_Q \cdot (e\otimes 1))^!$ the $Q\otimes Q-$dgmodule morphism
\[
(\Delta_Q \cdot (e\otimes 1))^!:s^{-(m+2k)}  Q \longrightarrow Q\otimes Q, s^{-(m+2k)} q \longmapsto \Delta_Q \cdot (e\otimes q),
\]
which is balanced.
Consider the mapping cone
\[
Q \otimes Q \bigoplus_{(\Delta_Q \cdot (1\otimes e))^!} ss^{-(m+2k)}Q
\]
which is a CDGA.

\begin{theo}
\label{T:ConfDxi}
With the notation above, assume that the vector bundle $\xi$ is of even rank $2k\geq 4$, and that the base, $M$, is a 2-connected closed manifold.  Then
\[
Q\otimes Q \bigoplus_{(\Delta_Q \cdot (1\otimes e))^!} ss^{-(m+2k)}Q
\]
is a CDGA model of $\Conf(D\xi,2)$.
\end{theo}

Before proving this theorem, let us first deduce the rational homotopy
invariance of that configuration space.
\begin{cor}
\label{C:invConfDxi}
The rational homotopy type of the configuration space of 2 points in a disk bundle of even rank $\geq 4$ over a 2-connected closed manifold depends only on the rational homotopy type of the base and on the Euler class.
\end{cor}

\begin{proof}[Proof of Corollary \ref{C:invConfDxi}]
By the main result of \cite{LaSt:PD}, the base of the bundle admits a
Poincar\'e duality CDGA model, $Q$.   Let $e\in Q \cap \ker(d_Q)$ be a
representative of the Euler class.   By Theorem \ref{T:ConfDxi}, a
CDGA model of the configuration space, 
and hence its rational homotopy type since it is simply connected, depends only on those data.
\end{proof}

\begin{proof}[Proof of Theorem \ref{T:ConfDxi}]
Denote by $\bar z$ a generator of degree $2k$ and define the CDGA
\[
P:=\left(\frac{Q\otimes \wedge \bar z}{(\bar z^2-e\bar z)}, D \bar z = 0\right)
\]
which is a Poincar\'e CDGA in dimension $n=m+2k$.  Define the CDGA morphism
\[
\varphi: P \longrightarrow Q
\]
by $\varphi (q_1+q_2 \bar z) = q_1+q_2 \cdot e$, for $q_1,q_2\in Q$.
 Then \cite[Theorem 4.1]{CLS:pretty} and its proof establishes that the pair $(D\xi,S\xi)$ admits a surjective pretty model associated to $\varphi$.   We will then use Theorem \ref{T:modele_FW2} to establish the model of $\Conf(D\xi,2)$.

Following the notation of \cite[proof of Theorem 4.1]{CLS:pretty}, one computes that
\[
I=\varphi^! (s^{-n}\# Q)=\Phi^! (s^{-2k}Q)=\bar z\cdot Q.
\]
We need to compute the truncated diagonal class $\overline \Delta \in
P/I \otimes P/I$.  Let $\{q_i\}$ be a homogeneous basis of $Q$ and let
$\{q^\ast_i\}$ be its
 Poincar\'e dual basis.  Denote by $\omega \in Q^m$ the fundamental class of $Q$, so that we have
\[
q_i\cdot q^\ast_j = \delta_{ij} \cdot \omega \mod Q^{<m}.
\]
Then
\begin{equation}
\label{E:basisP}
\{q_i\} \cup \{q_i\cdot \bar z\}
\end{equation}
is a homogeneous basis of $P$ and  $-\omega\bar z$ is a
fundamental class of $P$. Then the Poincar\'e dual basis of \eqref{E:basisP} is given by
\[
\{q^\ast_i \cdot (e-\bar z)\} \cup \{-q^\ast_i\}
\]
because of the four equations
\[
\begin{aligned}
q_i \cdot q^\ast_j (e-\bar z ) &= -\delta_{ij} \omega \bar z& \mod P^{<n},
\\
q_i\cdot (-q^\ast_j) &=0& \mod P^{<n},
\\
(q_i\bar z) \cdot (q^\ast_j (e-\bar z)) &= q_iq^\ast_j (\bar z e-\bar z^2) = 0& \mod P^{<n},
\\
(q_i \bar z) \cdot (- q^\ast_j) &= -q_i q^\ast_j \bar z = -\delta_{ij} \omega \bar z&\mod P^{<n}.
\end{aligned}
\]
Therefore the diagonal class in $P$ is given by
\[
\Delta_P = \sum_i (-1)^{|q_i|} (q_i \otimes q^\ast_i (e-\bar z)-q_i\bar z \otimes q^\ast_i)\in P\otimes P,
\]
and, since $I=Q\bar z$, the truncated diagonal class is
\[
\overline \Delta_P = \sum_i (-1)^{|q_i|} q_i \otimes q^\ast_i e \in P/I \otimes P/I.
\]
The diagonal class of $Q$ is
\[
\Delta_Q = \sum_i (-1)^{|q_i|} q_i \otimes q^\ast_i \in Q\otimes Q,
\]
therefore, using the canonical isomorphism $P/I \cong Q$, we have
\[
\overline \Delta_P = \Delta_Q \cdot (1\otimes e).
\]
The theorem is then a direct consequence of Theorem \ref{T:modele_FW2}.
\end{proof}

Note that the total space, $E\xi$, of the vector bundle $\xi$ is homeomorphic to the interior of $D\xi$, and therefore $\Conf(E\xi,2)\simeq \Conf(D\xi,2)$.   In particular, when the bundle is trivial the above gives a model of $\Conf(M\times \br^{2k},2)$.     Hence we recover partially the result \cite[Theorem1]{CoTa:BAMS}.

Interestingly enough\PL{check all the examples} we get different models when the bundle is not trivial.    Consider for example the quaternionic  Hopf line bundle, $\eta$, over $S^4$, of rank 4.
In that case we can take $Q=(\bq[x]/(x^2),d_Q=0)$, with $\deg(x)=4$, as a model for $S^4$
and the Euler class is represented by $e=x$.   Using the model of
Theorem \ref{T:ConfDxi}
 one computes easily that the rational cohomology algebra of $\Conf
 (D\eta,2)$ is the same as
 $H^\ast (S^4 \vee S^4 \vee S^{11};\bq)$, but $\Conf(D\eta,2)$ is
 \emph{not} formal because it admits  a non trivial Massey product in degree 11.

By contrast, for the trivial bundle of
rank $4$ over $S^4$, $\epsilon=S^4\times\BR^4$,  one computes that
$\Conf(D\epsilon,2)$ is formal and its rational cohomology algebra is given by
\[
H^\ast(\Conf(S^4\times \br^4,2);\bq) \cong \frac{\wedge(x,x',u)}{(x^2,x^{'2},ux-ux')}.
\]
with $\deg(x)=\deg(x')=4$ and $\deg(u)=7$.\PL{le cas $k=2$ n'est pas
  couvert par la prop. ce serait intéressant de construire un modèle
  de ces fibrés en cercles.}

Thus the two compact manifolds $D\eta$ and $D\epsilon$ of dimension
$8$ are homotopy equivalent but their configuration spaces have
different Poincar\'e series. This is because their boundaries,
$\partial D\eta=S^7$ and $\partial D\epsilon=S^4\times S^3$, are not
homotopy equivalent.

\subsection{A CDGA model for $\Conf(W,2)$ when $W$ is the complement of a subpolyhedron in a closed manifold}

Let $M$ be a 2-connected closed manifold of dimension $n$.  Let $K\subset M$ be a 2-connected subpolyhedron such that $\dim M \geq 2 \dim (K)+3$.  In this section we explain how to build a CDGA model of $\Conf(M\backslash K,2)$.

Let $T$ be a regular neighborhood of $K$ in $M$, in other words $T$ is
a compact codimension $0$ submanifold of $M$ that retracts by
deformation on $K$.   Then let $W$ be the closure of $ M\backslash T$
in $M$, which is a compact manifold with boundary $\partial
W= \partial T$.   The interior of $W$ is homeomorphic to $M\backslash
K$.    Therefore $\Conf(M\backslash K,2)$ is homotopy equivalent to
$\Conf(W,2)$. 

  Let us recall how to build a pretty surjective model of $(W,\partial W)$.
By \cite[Proposition 5.4]{CLS:pretty} one can construct a surjective CDGA model
\[
\varphi : P \twoheadrightarrow Q
\]
of $K \hookrightarrow M$, where $P$ is a Poincar\'e duality CDGA and
$Q^{\geq n/2-1}=0$.
As explained in \cite[end of Section 3]{CLS:pretty}, the main result of \cite{LaSt:PE} implies that the pretty model associated to $\varphi$,
\[
\varphi \oplus \id : P \bigoplus_{\varphi^!} ss^{-m} \# Q \longrightarrow Q \bigoplus_{\varphi\varphi^!} ss^{-n} \# Q
\]
is a CDGA model of $(W,\partial W)$.   Therefore, by Theorem \ref{T:modele_FW2}, a CDGA model of $\Conf(M\backslash K,2)$ is given by
\[
P/I \otimes P/I \bigoplus_{\overline \Delta^!} ss^{-n}(P/I).
\]

Let us illustrate this for the configuration space of a punctured
manifold.  Let $M$ be a closed 2-connected manifold and set
$W=M\backslash \{x_0\}$, with $x_0\in M$.   Let $P$ be a Poincar\'e duality CDGA model of $M$ with fundamental class
\[
\omega \in P^n.
\]
Pick a homogeneous basis $\{a_i\}_{0\leq i\leq N}$ of $P$ with $a_0=1$
and $a_N=\omega$. Let $\{a_i^*\}$ be the Poincar\'e dual basis. Then the
 diagonal class is
\[
\Delta = 1 \otimes \omega + (-1)^n \omega\otimes 1 + \sum^{N-1}_{i=1} a_i \otimes a^\ast_i
\]
 with $1\leq \deg (a_i) \leq n-1$, for $1\leq i\leq N-1$.   In that case we can take $I=\bq \cdot \omega$ and we get that
 \[
 \overline P = P/(\bq\cdot \omega)
 \]
is a CDGA model of $M\backslash \{x_0\}$ and the truncated diagonal is 
\[
\overline \Delta = \sum^{N-1}_{i=1} a_i \otimes a^\ast_i.
\]
Thus
\[
\overline P \otimes \overline P \bigoplus_{\overline \Delta^!} ss^{-n} \overline P
\]
is a CDGA model of $\Conf(M\backslash \{x_0\},2)$.

\bibliographystyle{amsplain}
\def\cprime{$'$}
\providecommand{\bysame}{\leavevmode\hbox to3em{\hrulefill}\thinspace}
\providecommand{\MR}{\relax\ifhmode\unskip\space\fi MR }
\providecommand{\MRhref}[2]{%
  \href{http://www.ams.org/mathscinet-getitem?mr=#1}{#2}
}
\providecommand{\href}[2]{#2}

\end{document}